\documentclass[11pt,A4]{amsart}
\usepackage{ifthen,latexsym,amssymb,amsmath}
\newtheorem{theorem}{Theorem}
\newtheorem{lemma}[theorem]{Lemma}

\newtheorem{conjecture}[theorem]{Conjecture}

\begin{document}

\author{Bart\l{}omiej Bzd\c{e}ga}

\address{Adam Mickiewicz University, Pozna\'n, Poland}

\email{exul@amu.edu.pl}



\title{On a generalization of Beiter Conjecture}

\maketitle

\begin{abstract}
We prove that for every $\varepsilon>0$ and a nonnegative integer $\omega$ there exist primes $p_1,p_2,\ldots,p_\omega$ such that for $n=p_1p_2\ldots p_\omega$ the height of the cyclotomic polynomial $\Phi_n$ is at least $(1-\varepsilon)c_\omega M_n$, where $M_n=\prod_{i=1}^{\omega-2}p_i^{2^{\omega-1-i}-1}$ and $c_\omega$ is a constant depending only on $\omega$; furthermore $\lim_{\omega\to\infty}c_\omega^{2^{-\omega}}\approx0.71$. In our construction we can have $p_i>h(p_1p_2\ldots p_{i-1})$ for all $i=1,2,\ldots,\omega$ and any function $h:\mathbb{R}_+\to\mathbb{R}_+$.
\end{abstract}

\section{Introduction}

Let $\Phi_n$ be the $n$th cyclotomic polynomial, i.e. the unique monic polynomial irreducible over integers, which roots are all primitive $n$th roots of unity. We assume that $n=p_1p_2\ldots p_\omega$ and $2 < p_1 < p_2 < \ldots < p_\omega$ are primes, since $\Phi_{2n}(x)=\Phi_n(-x)$ for odd $n$ and $\Phi_{np}(x)=\Phi_n(x^p)$ for a prime $p$ dividing $n$. We call the number $\omega=\omega(n)$ the order of $\Phi_n$.

Let $A_n$ denotes the maximal absolute value of a coefficient of $\Phi_n$. We say shortly that $A_n$ is the height of $\Phi_n$. In case of $\omega\in\{0,1,2\}$ determining of $A_n$ is easy and we have $A_1=A_{p_1}=A_{p_1p_2}=1$. For $\omega=3$ it is known that $A_{p_1p_2p_3}\le\frac34p_1$ \cite{Bachman-TernaryBounds}. The Corrected Beiter Conjecture states that $A_{p_1p_2p_3}\le\frac23p_1$ (see \cite{GallotMoree-BeiterCounter} and references given there for details). The constant $\frac23$ is best possible if the conjecture is true.

For cyclotomic polynomials of any order we put
$$M_n = \prod_{i=1}^{\omega-2}p_i^{2^{\omega-1-i}-1},$$
where the empty product, which happens if $\omega\le2$, equals $1$. P.T.~Bateman, C.~Pomerance and R.C.~Vaughan proved in \cite{BatemanPomeranceVaughan-Size} that $A_n\le M_n$. In \cite{Bzdega-Height} the author proved that $A_n\le C_\omega M_n$, where $C_\omega^{2^{-\omega}}$ converges to approximately $0.95$ with $\omega\to\infty$. However, so far we have known no good general class of $\Phi_n$ for which $A_n$ is close to $C_\omega M_n$.

It has not been even known if $M_n$ gives the optimal order for the upper bound on $A_n$. For example we have $A_{p_1\ldots p_5}\le C_5 p_1^7p_2^3p_3$, but we did not know whether $A_{p_1\ldots p_5}\le C_5' p_1^8p_2^2p_3$ for some other constant $C_5'$. All known constructions of $\Phi_n$ with large height required that most prime factors of $n$ are of almost the same size.

One of the main purposes of this paper is to show that $M_n$ is optimal, i.e. in the upper bound on $A_n$ it cannot be replaced by any smaller product of the form $p_1^{\alpha_1}p_2^{\alpha_2}\ldots p_\omega^{\alpha_\omega}$ in a sense which we describe below.

For a fixed $\omega$ we define the following strict lexicographical order on $\mathbb{R}^\omega$:
\begin{align*}
& (\alpha_1,\alpha_2,\ldots,\alpha_\omega)\prec(\beta_1,\beta_2,\ldots,\beta_\omega) \\
\iff & \alpha_\omega=\beta_\omega,\alpha_{\omega-1}=\beta_{\omega-1},\ldots,\alpha_{k+1}=\beta_{k+1} \text{ and } \alpha_k<\beta_k \text { for some } k\le\omega.
\end{align*}

For $\alpha=(\alpha_1,\alpha_2,\ldots,\alpha_\omega)$ and $n=p_1p_2\ldots p_\omega$ we put $M_n^{(\alpha)}=p_1^{\alpha_1}p_2^{\alpha_2}\ldots p_\omega^{\alpha_\omega}$. Note that if $\alpha\prec\beta$ and $p_i$ is large enough compared to $p_1p_2\ldots p_{i-1}$ for all $i\le\omega$, then $M_n^{(\alpha)} < M_n^{(\beta)}$.

Therefore, we say that $M_n^{(\alpha)}$ is the optimal bound on $A_n$ for a fixed $\omega$ if there exists a constant $b_\omega$ such that $A_n\le b_\omega M_n^{(\alpha)}$ for all $n$ with $\omega(n)=\omega$ and $\alpha$ is smallest possible in sense of the order $\prec$.

It requires an explanation what it means that $p_i$ is large enough compared to $p_1p_2\ldots p_{i-1}$ for all $i\le\omega$. Let $h:\mathbb{R}_+\to\mathbb{R}_+$ be any function, preferably growing fast. We say that a sequence of primes $p_1,p_2,\ldots,p_\omega$ is $h$-\emph{growing} if $p_i\ge h(p_1p_2\ldots p_{i-1})$ for $i=1,2,\ldots,\omega$ (empty product equals $1$). With a small abuse of notation we will also write that the number $n=p_1p_2\ldots p_\omega$ is $h$-growing.

The following theorem is the main result of this paper.

\begin{theorem} \label{T large An}
For every $\omega\ge3$, $\varepsilon>0$ and $h:\mathbb{R}_+\to\mathbb{R}_+$ there exists an $h$-growing $n=p_1p_2\ldots p_\omega$ such that $A_n>(1-\varepsilon)c_\omega M_n$, where
$$M_n = \prod_{i=1}^{\omega-2}p_i^{2^{\omega-1-i}-1} \quad\text{and}\quad c_\omega = \frac1\omega\cdot\left(\frac2\pi\right)^{3\cdot2^{\omega-3}}\cdot\left(\prod_{k=3}^{\omega-1}k^{2^{\omega-1-k}}\right)^{-1}.$$
\end{theorem}

By this theorem and the already mentioned result from \cite{Bzdega-Height}, $M_n$ is the optimal bound on $A_n$. Furthermore
$$\lim_{\omega\to\infty}c_\omega^{2^{-\omega}} = \left(\frac2\pi\right)^{3/8} \cdot \prod_{k=3}^\infty k^{-2^{-k-1}} \approx 0.71.$$
Let us define the $\omega$th Beiter constant in the following natural way:
$$B_\omega = \limsup_{\omega(n)=\omega}(A_n/M_n).$$
For example we know that $B_0=B_1=B_2=1$ and $\frac23\le B_3 \le \frac34$. If Corrected Beiter Conjecture is true, then $B_3=\frac23$.

For all $\omega$ we have
$$c+o(1) < B_\omega^{2^{-\omega}} < C+o(1), \quad \omega\to\infty$$
with $c\approx 0.71$ and $C\approx 0.95$. It would be interesting to know the asymptotics of $B_\omega$. For example, we expect that the following natural conjecture is true.

\begin{conjecture}\label{C BeiterGeneral}
There exists a limit $\lim_{\omega\to\infty}B_\omega^{2^{-\omega}}$.
\end{conjecture}

\section{Preliminaries and binary case}

Let us define the value
$$L_n = \max_{|z|=1}|\Phi_n(z)|.$$
It was already considered by several authors \cite{BatemanPomeranceVaughan-Size,KonyaginMaierWirsing-ManyPrimes,Maier-Anatomy} while estimating $A_n$. If $S_n$ denotes the sum of absolute values of the coefficients of $\Phi_n$, then for $n>1$
$$A_n \ge \frac{S_n}{\deg\Phi_n+1} \ge \frac{L_n}n.$$

We express $|\Phi_n(z)|$ as a real function of $x=\arg(z)$ for $|z|=1$. For all $n\ge1$ let
$$F_n(x) = \prod_{d\mid n}\left(\sin\frac d2x\right)^{\mu(n/d)},$$
where we put $\frac{\sin ax}{\sin bx} = \frac ab$ for $\sin bx=\sin ax = 0$. Note that $F_n$ is periodic with the period $2\pi$. By the following lemma $F_n(x)$ is well defined for all $x\in\mathbb{R}$.

\begin{lemma} \label{L |z|=1}
For $n>1$ we have $|\Phi_n(e^{ix})|=|F_n(x)|$.
\end{lemma}

\begin{proof}
By elementary computations $|1-z| = 2\left|\sin\frac12x\right|$. Then we use the well known Moebius formula $\Phi_n(z)=\prod_{d\mid n}(1-z^d)^{\mu(n/d)}$. Note that $\Phi_n(e^{ix})$ is a bounded continous function of $x$, so if the product $F_n(x_0)$ is not defined for some $x_0$ (which happens only for finitely many values of $0\le x_0<2\pi$), then we can replace it by its limit with $x\to x_0$.
\end{proof}

By Lemma \ref{L |z|=1} we have
$$L_n = \max_{|z|=1}|\Phi_n(z)| = \max_{0\le x < 2\pi}|F_n(x)|$$
as long as $n>1$. Furthermore $|F_1(x)|=\frac12|\Phi_1(e^{ix})|$.

It is easy to determine $L_1=1$ and $L_{p_1}=p_1$. Let us consider the case $\omega=2$.

\begin{theorem} \label{T binary}
Let $p_1<p_2$ be primes and let $a$ be the unique integer such that $p_1\mid p_2+2a$ and $|a|<p_1/2$. Then $L_{p_1p_2}\ge\frac{4(p_1-2)p_2}{\pi^2|2a+1|}$.
\end{theorem}

\begin{proof}
Put $x=\left(1+\frac1{p_1}+\frac{2a+1}{p_1p_2}\right)\pi$. Then
\begin{align*}
\left|\sin \frac{p_1p_2x}2\right| & = \left|\sin\frac{p_1p_2+p_2+2a+1}2\pi\right| = 1,\\
\left|\sin \frac x2\right| & = \left|\cos\left(\frac{1}{2p_1}+\frac{2a+1}{2p_1p_2}\right)\pi\right| \ge 1-\frac1p_1-\frac{|2a+1|}{p_1p_2} \ge 1-\frac2{p_1},
\end{align*}
where we used the inequality $\cos t \ge 1-\frac2\pi\cdot |t|$ for $|t| \le \pi/2$. Furthermore
\begin{align*}
\left|\sin \frac{p_1x}2\right| & = \left|\sin\left(\frac{p_1+1}2+\frac{2a+1}{2p_2}\right)\pi\right| = \left|\sin\frac{2a+1}{2p_2}\pi\right| \le \frac{|2a+1|\pi}{2p_2},\\
\left|\sin \frac{p_2x}2\right| & = \left|\sin\left(\frac{p_2}2+\frac{p_2+2a}{2p_1}+\frac1{2p_1}\right)\pi\right| = \left|\sin\frac\pi{2p_1}\right| \le \frac\pi{2p_1},
\end{align*}
where we used the inequality $|\sin t| \le |t|$ for $t\in\mathbb{R}$. By the above inequalities we obtain
$$L_{p_1p_2} \ge F_{p_1p_2}(x) = \left|\frac{\sin(x/2)\sin(p_1p_2x/2)}{\sin(p_1x/2)\sin(p_2x/2)}\right| \ge \frac{4(p_1-2)p_2}{\pi^2|2a+1|},$$
as desired.
\end{proof}

\section{Derivative of $F_n$}

It is not difficult to prove that $F_n$ is a differentiable function. Let $f_n(x)$ be the derivative of $F_n(x)$. The function $f_n$ plays a crucial role in our construction of $n$ with large $L_n$, especially its minimal absolute values in points $x_0$ for which $F_n(x_0)=0$. Let
$$D_n = \min_{x_0: \; F_n(x_0)=0}|f_n(x_0)|.$$
The aim of this section is to prove the following theorem.

\begin{theorem} \label{T Dn formula}
For all positive integers $\omega$ and all $\varepsilon>0$ there exists a function $h_{\omega,\varepsilon}:\mathbb{R}_+\to\mathbb{R}_+$ depending only on $\omega$ and $\varepsilon$, such that
$$\frac n2 \cdot (L_{p_1}L_{p_1p_2}\ldots L_{p_1p_2\ldots p_{\omega-1}})^{-1} \le D_n < (1+\varepsilon)\frac n2 \cdot (L_{p_1}L_{p_1p_2}\ldots L_{p_1p_2\ldots p_{\omega-1}})^{-1}$$
for all $h_{\omega,\varepsilon}$-growing $n=p_1p_2\ldots p_\omega$.
\end{theorem}

In order to prove this theorem we will need some lemmas.

\begin{lemma} \label{L f}
If $F_n(x_0)=0$, then
$$|f_n(x_0)| = \frac n2 \prod_{d\mid n, \; d\neq n}\left|\sin\frac d2x_0\right|^{\mu(n/d)}.$$
\end{lemma}

\begin{proof}
Since $x_0=\frac{2t_0\pi}{n}$ with some integer $t_0$ coprime to $n$, we have
\begin{align*}
f_n(x_0) & = \lim_{\epsilon\to\infty}\frac1\epsilon\prod_{d\mid n}\left(\sin\frac d2(x_0+\epsilon)\right)^{\mu(n/d)} \\
& = \lim_{\epsilon\to\infty}\frac{\sin(t_0\pi+n\epsilon/2)}{\epsilon}\prod_{d\mid n, \; d\neq n}\left(\sin\frac d2(x_0+\epsilon)\right)^{\mu(n/d)} \\
& = \pm\frac n2\prod_{d\mid n, \; d\neq n}\left(\sin\frac d2x_0\right)^{\mu(n/d)},
\end{align*}
as desired.
\end{proof}

\begin{lemma} \label{L fnp}
Let $p$ be a prime not dividing $n$. If $F_{np}(x_1)=0$, then $f_{np}(x_1) = \frac{p|f_n(x_1p)|}{|F_n(x_1)|}$.
\end{lemma}

\begin{proof}
By Lemma \ref{L f}
\begin{align*}
|f_{np}(x_1)| & = \frac{np}2 \prod_{d\mid np, \; d\neq np}\left|\sin\frac d2x_1\right|^{\mu(np/d)} \\
& = \frac{np}2\cdot\left(\prod_{d\mid n}\left|\sin\frac d2x_1\right|^{\mu(n/d)}\right)^{-1}\cdot\left(\prod_{d\mid n, \; d\neq n}\left|\sin\frac{dp}2x_1\right|^{\mu(n/d)}\right) \\
& = \frac{np}2\cdot|F_n(x_1)|^{-1}\cdot\frac2n\cdot|f_n(px_1)| = \frac{p|f_n(px_1)|}{|F_n(x_1)|},
\end{align*}
which completes the proof.
\end{proof}

\begin{lemma} \label{L Dnp}
We have $D_{np} \ge p\cdot\frac{D_n}{L_n}$. Moreover, for all $\varepsilon>0$ there exists a function $h_\varepsilon:\mathbb{R}_+\to\mathbb{R}_+$ depending only on $\varepsilon$, such that $D_{np} < (1+\varepsilon)\cdot p\cdot\frac{D_n}{L_n}$ for all $p>h_\varepsilon(n)$.
\end{lemma}

\begin{proof}
Let $x_0$ and $x_1$ be such that $F_n(x_0)=F_{np}(x_1)=0$, $|f_n(x_0)|=D_n$ and $|f_{np}(x_1)|=D_{np}$. Since $x_1=\frac{2t_1\pi}{np}$ for some $t_1$ coprime to $np$, we have $px_1=\frac{2t_1\pi}{n}$. Therefore $F_n(px_1)=0$ and hence $|f_{p}(px_1)|\ge D_n$. By applying this inequality and Lemma \ref{L fnp} we obtain
$$D_{np}=|f_{np}(x_1)|=\frac{p|f_n(px_1)|}{|F_n(x_1)|} \ge p\cdot\frac{D_n}{L_n}.$$

For obtaining the opposite inequality, let $x_0=\frac{2t_0\pi}n$ and $x_1'=\frac{x_0+2t\pi}p = \frac{2(t_0+tn)\pi}{np}$ with any $t\not\equiv-\frac{t_0}n\pmod p$. Then $F_{np}(x_1')=0$ and $f_n(px_1')=D_n$. Again by Lemma \ref{L fnp}
$$D_{np} \le |f_{np}(x_1')| = \frac{p|f_n(px_1')|}{|F_n(x_1')|} = p\cdot\frac{D_n}{\left|F_n\left(\frac{x_0+2t\pi}p\right)\right|}.$$
By choosing an appropriate $t$ we can have $\left|F_n\left(\frac{x_0+2t\pi}p\right)\right|$ as close to $L_n$ as we wish when $p\to\infty$.
\end{proof}

Now we are ready to prove the main theorem of this section.

\begin{proof}[Proof of Theorem \ref{T Dn formula}]
Let $\varepsilon>0$ be fixed and let $\varepsilon'=\sqrt[\omega]{1+\varepsilon}-1$. Let $h_{\varepsilon'}$ be a function given by Lemma \ref{L Dnp}, which implies that if $n=p_1p_2\ldots p_\omega$ is $h_{\varepsilon'}$-growing, then
$$p_i\cdot\frac{D_{p_1p_2\ldots p_{i-1}}}{L_{p_1p_2\ldots p_{i-1}}} \le D_{p_1p_2\ldots p_i} < (1+\varepsilon')p_i\cdot\frac{D_{p_1p_2\ldots p_{i-1}}}{L_{p_1p_2\ldots p_{i-1}}}$$
for $i=1,2\ldots,\omega$ (empty product equals $1$). By these inequalities
$$\frac{nD_1}{L_1L_{p_1}L_{p_1p_2}\ldots L_{p_1p_2\ldots p_{\omega-1}}} \le D_n < (1+\varepsilon')^\omega\frac{nD_1}{L_1L_{p_1}L_{p_1p_2}\ldots L_{p_1p_2\ldots p_{\omega-1}}}.$$
Note that $(1+\varepsilon')^\omega=1+\varepsilon$, $L_1=1$ and $D_1=\frac12$. So the theorem holds with the function $h_{\omega,\varepsilon} = h_{\varepsilon'} = h_{\sqrt[\omega]{1+\varepsilon}-1}$, which clearly depends only on $\omega$ and $\varepsilon$.
\end{proof}

\section{Proof of main result}

In the following lemma we give a lower bound on $L_{np}$ which depends on the residue class of $p$ modulo $n$.

\begin{lemma} \label{L Lnp}
Let $\varepsilon>0$ and $n=p_1p_2\ldots p_\omega$ be fixed. Put $x_M\in[0,2\pi)$ such that $F_n(x_M)=L_n$ and $x_0=\frac{2t_0\pi}{n}$ for which $F_n(x_0)=0$ and $|f_n(x_0)|=D_n$. Let $b=\min_{k\in\mathbb{Z}}\left|\frac{nx_M}{2\pi}-pt_0+nk\right|$. Then
$$L_{np}>(1-\varepsilon)L_n\cdot\frac{np}{2b\pi D_n}$$
for every $p$ large enough. Furthermore, if $p_1>\omega$ and $r$ is an integer coprime to $n$ such that $\left|\frac{nx_M}{2\pi}-r\right|$ is smallest possible, then
$$L_{np}>(1-\varepsilon)L_n\cdot\frac{1}{\pi(\omega+1)}\cdot\frac{np}{D_n}$$
for every sufficiently large $p\equiv \frac r{t_0} \pmod n$.
\end{lemma}

\begin{proof}
We have $F_n(x)=\frac{F_n(px)}{F_n(x)}$, so
$$L_{np} = \max_{0\le x < 2\pi}\left|\frac{F_n(px)}{F_n(x)}\right| \ge \max_{k\in\mathbb{Z}}\frac{|F_n(x_M+2k\pi)|}{\left|F_n\left(\frac{x_M+2k\pi}{p}\right)\right|} = \frac{L_n}{\min_{k\in\mathbb{Z}}\left|F_n\left(\frac{x_M+2k\pi}{p}\right)\right|}.$$
Let $k_0$ be a integer for which $\left|\frac{x_M+2k_0\pi}{p}-x_0\right|$ is smallest possible. Then
\begin{align*}
\min_{k\in\mathbb{Z}}\left|F_n\left(\frac{x_M+2k\pi}{p}\right)\right| & \le \left|F_n\left(\frac{x_M+2k_0\pi}{p}\right)\right| \\
& \sim |f_n(x_0)|\cdot\left|\frac{x_M+2k_0\pi}{p}-x_0\right| \quad (\text{with } p\to\infty) \\
& = D_n\cdot\frac{2\pi}{np}\cdot\left|\frac{nx_M}{2\pi}-t_0p+k_0n\right| \\
& = D_n\cdot\frac{2b\pi}{np}.
\end{align*}
Therefore
$$L_{np} > (1+o(1))\frac{L_n}{D_n\cdot\frac{2b\pi}{np}} \sim L_n\cdot\frac{np}{2b\pi D_n}$$
with $p\to\infty$, which completes the proof of the first statement.

For $p\equiv \frac r{t_0} \pmod n$ we have
$$b=\min_{k\in\mathbb{Z}}\left|\frac{nx_M}{2\pi}-pt_0+nk\right|=\left|\frac{nx_M}{2\pi}-r\right|\le\frac{\omega+1}2$$
since, in view of $p_1>\omega$, at most $\omega$ consecutive integers are not coprime to $p$.
\end{proof}

Simple calculations show that Theorem \ref{T binary} gives a better lower bound for $L_{p_1p_2}$ than Lemma \ref{L Lnp}. Therefore we use Theorem \ref{T binary} in the proof of the main result. By the fact that $A_n\ge L_n/n$ for $n>1$, Theorem \ref{T large An} is an immediate consequence of the following theorem.

\begin{theorem} \label{T large Ln}
For every $\omega\ge3$, $\varepsilon>0$ and $h:\mathbb{R}_+\to\mathbb{R}_+$ there exists an $h$-growing $n=p_1p_2\ldots p_\omega$ such that $L_n>(1-\varepsilon)c_\omega nM_n$, where $c_\omega$ and $M_n$ are defined in Theorem \ref{T large An}.
\end{theorem}

\begin{proof}
We prove this by a strong induction on $\omega=\omega(n)$. The induction starts with $\omega=2$.

Our inductive assumption is that for all $\varepsilon'>0$ and a function $h:\mathbb{R}_+\to\mathbb{R}_+$ there exists an $h$-growing $n=p_1p_2\ldots p_\omega$ such that $L_{p_1p_2}>(1-\varepsilon')\frac4{\pi^2}p_1p_2$ and $L_{p_1p_2\ldots p_i}>(1-\varepsilon')c_ip_1p_2\ldots p_iM_{p_1p_2\ldots p_i}$ for $3\le i \le \omega$. By Theorem \ref{T binary} it is true for $\omega=2$ with $p_1\mid q_1-2$ (note that the second part of the inductive assumption is empty when $\omega=2$).

Now we show the inductive step. Let $\omega\ge2$. Without loss of generality we may assume that $h(1)\ge\omega$. By Lemma \ref{L Lnp} and Dirichlet's theorem on primes in arithmetic progressions, there exists $p_{\omega+1}>h(p_1p_2\ldots p_\omega)$ for which
$$L_{p_1p_2\ldots p_{\omega+1}} > (1-\varepsilon')L_n\cdot\frac{np_{\omega+1}}{\pi(\omega+1)D_n}.$$

By Theorem \ref{T Dn formula} there exists a function $h_1:\mathbb{R}_+\to\mathbb{R}_+$ depending only on $\omega$ and $\varepsilon'$, such that for all $h_1$-growing $n$
$$D_n > (1-\varepsilon')^{-1}\frac n2 \cdot \frac1{L_{p_1}L_{p_1p_2}\ldots L_{p_1p_2\ldots p_\omega}}.$$
Again without loss of generality we can assume that $h(x)>h_1(x)$ for all $x\in\mathbb{R}_+$. In this situation all $h$-growing numbers are also $h_1$-growing, so the above inequality holds for every $h$-growing $n$.

For given $\varepsilon>0$ we choose $\varepsilon'=1-\sqrt[\omega+1]{1-\varepsilon}$. By the above inequalities and the inductive assumption
\begin{align*}
L_{p_1p_2\ldots p_{\omega+1}} & > (1-\varepsilon')^2\cdot\frac{2p_{\omega+1}}{\pi(\omega+1)}\cdot L_{p_1}L_{p_1p_2}\ldots L_{p_1p_2\ldots p_\omega} \\
& > (1-\varepsilon')^{\omega+1} \cdot\frac{2p_{\omega+1}}{\pi(\omega+1)}\cdot p_1 \cdot \frac4{\pi^2}p_1p_2 \cdot \prod_{i=3}^\omega(c_ip_1p_2\ldots p_iM_{p_1p_2\ldots p_i}) \\
& = (1-\varepsilon)\left(\frac8{\pi^3(\omega+1)}\cdot\prod_{i=3}^\omega c_i\right)\left(p_{\omega+1}\prod_{i=1}^\omega(p_1p_2\ldots p_iM_{p_1p_2\ldots p_i})\right).
\end{align*}
The exponent of $p_k$ in $\prod_{i=1}^\omega(p_1p_2\ldots p_iM_{p_1p_2\ldots p_i})$ for $k\le\omega$ equals
$$\omega-k+1 + \sum_{i=k+2}^\omega(2^{i-k-1}-1) = 2^{\omega-k},$$
so
$$p_{\omega+1}\prod_{i=1}^\omega(p_1p_2\ldots p_iM_{p_1p_2\ldots p_i})= p_1p_2\ldots p_{\omega+1}M_{p_1p_2\ldots p_{\omega+1}}.$$

It remains to evaluate the constant by using a similar method:
\begin{align*}
\frac8{\pi^3(\omega+1)}\cdot\prod_{i=3}^\omega c_i & = \frac8{\pi^3(\omega+1)}\cdot\prod_{i=3}^\omega\left(\frac1i\cdot\left(\frac2\pi\right)^{3\cdot2^{i-3}}\cdot\left(\prod_{k=3}^{i-1}k^{2^{i-1-k}}\right)^{-1}\right) \\
& = \frac1{\omega+1}\cdot\left(\frac2\pi\right)^{3\cdot2^{\omega-2}}\cdot\frac1{3\cdot4\cdot\ldots\cdot\omega}\cdot\left(\prod_{i=3}^\omega\prod_{k=3}^{i-1}k^{2^{i-1-k}}\right)^{-1} \\
& = \frac1{\omega+1}\cdot\left(\frac2\pi\right)^{3\cdot2^{\omega+1-3}} \cdot \left(\prod_{t=3}^{\omega+1-1}t^{2^{\omega+1-1-t}}\right)^{-1} = c_{\omega+1}
\end{align*}
for $\omega\ge2$.

Note that $\varepsilon'<\varepsilon$, so by the inductive assumption also $L_{p_1p_2}>(1-\varepsilon)\frac4{\pi^2}p_1p_2$ and $L_{p_1p_2\ldots p_i}>(1-\varepsilon)c_ip_1p_2\ldots p_iM_{p_1p_2\ldots p_i}$ for all $3\le i \le \omega$. It completes the inductive step.
\end{proof}

\section*{Acknowledgment}

The author is partially supported by NCN grant no.~2012/07/D/ST1/02111.

The author would like to thank Pieter Moree for the invitation to Max Planck Institute in Bonn and for the good time spent there.

\end{document}